\pdfoutput=1
\RequirePackage{ifpdf}
\ifpdf 
\documentclass[pdftex]{sigma}
\else
\documentclass{sigma}
\fi

\numberwithin{equation}{section}

\newtheorem{Theorem}{Theorem}[section]
\newtheorem*{Theorem*}{Theorem}
\newtheorem{Corollary}[Theorem]{Corollary}
\newtheorem{Lemma}[Theorem]{Lemma}

 { \theoremstyle{definition}
\newtheorem{Definition}[Theorem]{Definition}

\newtheorem{Remark}[Theorem]{Remark}
\newtheorem{asp}[Theorem]{Assumption} }

\usepackage{overpic}
\usepackage{bm}
\newcommand{\wick}[1]{\textrm{\textup{{\Large :}}}#1 \textrm{\textup{\Large :}}}

\begin{document}
\allowdisplaybreaks

\newcommand{\arXivNumber}{2104.13751}

\renewcommand{\PaperNumber}{002}

\FirstPageHeading

\ShortArticleName{Voros Coefficients at the Origin and at the Infinity}

\ArticleName{Voros Coefficients at the Origin and at the Infinity\\ of the Generalized Hypergeometric Differential\\ Equations with a Large Parameter}

\Author{Takashi AOKI~$^{\rm a}$ and Shofu UCHIDA~$^{\rm b}$}

\AuthorNameForHeading{T.~Aoki and S.~Uchida}

\Address{$^{\rm a)}$~Department of Mathematics, Kindai University, Higashi--Osaka 577-8502, Japan}
\EmailD{\href{mailto:aoki@math.kindai.ac.jp}{aoki@math.kindai.ac.jp}}

\Address{$^{\rm b)}$~Graduate School of Science and Engineering, Kindai University,\\
\hphantom{$^{\rm b)}$}~Higashi--Osaka 577-8502, Japan}
\EmailD{\href{mailto:1944310104r@kindai.ac.jp}{1944310104r@kindai.ac.jp}}

\ArticleDates{Received July 20, 2021, in final form December 30, 2021; Published online January 03, 2022}

\Abstract{Voros coefficients of the generalized hypergeometric differential equations with a large parameter are defined and their explicit forms are given for the origin and for the infinity. It is shown that they are Borel summable in some specified regions in the space of parameters and their Borel sums in the regions are given.}


\Classification{33C20; 34E20; 34M60}

\section{Introduction}

The notion of the Voros coefficients is one of the keys in the exact WKB analysis of
differential equations with a large parameter. It has been introduced mainly for
second-order ordinary differential equations and used effectively
in the descriptions of the parametric Stokes phenomena, of calculation of the monodromies and of the relations between Borel resummed WKB solutions and
classical special functions \cite{ATT,ATT2,ATT4,ATT3,AT,DDP,IKo,KoT,SS,T,V}.
Recently, new insights into the relationship between cluster algebra and Voros coefficients have been obtained in~\cite{IN}.
Furthermore, it is known that Voros coefficient of the Gauss hypergeometric differential equation can be described using the free energy in the topological recursion~\cite{ IKoYu}.
Besides, the Voros coefficients have also appeared in the physics literature as the spectral coordinates~\cite{HK}.

Our purpose is to define the Voros coefficients at the origin and at the infinity
for the generalized hypergeometric differential equation with a large parameter and to give the explicit forms of them.
Here the generalized hypergeometric differential equation means the differential equation which characterizes the generalized hypergeometric series (or function) ${}_{N}F_{N-1}$.
It is well known \cite{EMOT} that this equation is of order $N$ and obtained as a natural generalization of the Gauss hypergeometric differential equation.
The monodromy group of the generalized hypergeometric differential equation plays a role in studying the solution of the equation.
The classification of the finite hypergeometric groups \cite{L} and the differential Galois group of the generalized hypergeometric equation have been obtained in~\cite{BH}.

Basic notions and tools in the exact WKB analysis of higher-order or of infinite-order differential equations with a large parameter are established in \cite{AKKoT}. Our discussions are based on the theory developed in this paper.
For example, WKB solutions of a higher-order differential equation are defined
by using the characteristic roots of the equation, which are branches of an algebraic function.
Once WKB solutions are introduced, one can define the Voros coefficients in a similar manner to the second-order case.
Fortunately, our equation has ladder operators for parameters as in the case of the Gauss hypergeometric differential equation.
Therefore the basic idea of computation of the explicit forms of the Voros coefficients is the same as the second-order case~\cite{ATT3,AT,KoT,T}.
But in our case, we have to consider the integral of the algebraic function.
Although this fact causes some difficulties, some of them can be overcomed by using an idea invented by Iwaki and Koike~\cite{IKo}.
There is another difficulty, or complexity, coming from the number of parameters. Our equation contains $(2N-1)$ parameters and we have to manage some combinations of them. Thus our descriptions become somewhat complicated.

The plan of this article is as follows. In Section~\ref{section2}, we define WKB solutions of the
generalized hypergeometric differential equation with a large parameter and
investigate the local behavior of the solutions at the singularities of the equation.
In Section~\ref{section3}, we define the Voros coefficients at the origin and at the infinity for our equation. The explicit forms of the Voros coefficients are given (Theorem~\ref{thm1}).

\section{WKB solutions}\label{section2}
\subsection[The generalized hypergeometric differential equation with a large parameter]{The generalized hypergeometric differential equation\\ with a large parameter}

Let $N$ be an integer greater than $1$,
let $\bm{a}$ denote an $N$-tuple of parameters $a_i \in \mathbb{C}$, $i=1,2,\dots,N$,) and let $\bm{b}$ denote an $(N-1)$-tuple of parameters $b_j \in \mathbb{C}\setminus \{0,-1,-2, \dots\}$, $j=1,2,\dots,N-1$. We set
\begin{gather}\label{NFN-1}
{}_NF_{N-1}\left(\begin{matrix}\bm{a} \\ \bm{b}\end{matrix};x \right)=\sum_{k=0}^{\infty} \frac{(a_1)_k (a_2)_k\cdots (a_N)_k}{(b_1)_k (b_2)_k\cdots (b_{N-1})_k k !} x^k
\end{gather}
and call this the generalized hypergeometric series~\cite{EMOT}. Here $(\alpha)_k$ denotes the Pochhammer symbol $(={\Gamma(\alpha+k)}/{\Gamma(\alpha)})$. As is well known, the radius of convergence of \eqref{NFN-1} equals $1$ and the right-hand side of \eqref{NFN-1} defines a holomorphic function on the universal covering of $\mathbb{C}\setminus \{0,1\}$, which is also denoted by ${}_NF_{N-1}$.
This series or function satisfy the following $N$-th order ordinary differential equation:
\begin{gather}\label{NFN-1eq}
\left(\left(\prod_{j=1}^{N-1}(\vartheta_x+b_j)\right) \partial_x- \left(\prod_{i=1}^{N}(\vartheta_x+a_i) \right)\right)w=0.
\end{gather}
Here we set $\partial_x=\frac{{\rm d}}{{\rm d}x}$ and $\vartheta_x=x\frac{{\rm d}}{{\rm d}x}$.
We call \eqref{NFN-1eq} the generalized hypergeometric differential equation.
This equation has regular singular points at $x=0,1,\infty$.
We set $b_N=1$, $\tilde{\bm{b}}=(b_1, b_2,\dots,b_{N-1}, 1)$ and $\bm{1}=(1,1,1,\dots,1)$.
We suppose that $a_i-a_j, b_i-b_j \not\in \mathbb{Z}$, $1\leq i, j\leq N$, $i\not=j$, and $\sum_{i=1}^{N}{(b_i-a_i)}\not\in\mathbb{Z}$. Then there exists a fundamental system of solutions around the singular point $\big\{ w_1^{[\varrho]},\dots ,w_N^{[\varrho]}\big\}$ of~\eqref{NFN-1eq}. If $\varrho=0,\infty$, they are given by
\begin{gather*}
w_{1}^{[0]}={}_NF_{N-1}\left(\begin{matrix}\bm{a} \\ \bm{b}\end{matrix};x \right),\\
w_{j+1}^{[0]}=x^{1-b_j} {}_NF_{N-1}\left(\begin{matrix}\bm{a}+(1-b_j)\bm{1} \\ \big( \tilde{\bm{b}}+(1-b_j)\bm{1} \big)^{\lor} \end{matrix};x \right),\\
w_{i}^{[\infty]}=(-x)^{-a_i} {}_NF_{N-1}\left(\begin{matrix}(1+a_i)\bm{1}-\tilde{\bm{b}}\\ ((1-a_{i})\bm{1}+\bm{a})^{\lor}\end{matrix};\frac{1}{x} \right)
\end{gather*}
for $j=1,\dots, N-1$, $i=1,\dots, N$, where $\ast^{\lor}$ indicates that the entry $1+b_{j}-b_{j}$ or $1-a_{i}+a_{i}$ omitted in~$\ast$~\cite{O}.
Connection problem of~\eqref{NFN-1eq} between $x=0$ and $x=\infty$ had already been studied well~\cite{B,KN,S}.

We introduce a positive large parameter $\eta$ in $\bm{a}$ and $\bm{b}$ by setting
\begin{gather*}
a_i={a_{i,0}}+a_{i,1}{\eta},\qquad i=1,2,\dots,N,\\
b_j={b_{j,0}}+b_{j,1}{\eta},\qquad j=1,2,\dots,N-1,
\end{gather*}
where $a_{i,k}, b_{j,k}\in\mathbb{C}$ for $k=0,1$. We consider the following equation which contains the large parameter:
\begin{gather} \label{nFn-1eq}
{}_N P_{N-1}\psi =0,
\end{gather}
where we set
\begin{gather}\label{nPn-1}
{}_N P_{N-1}=
\eta^{-N}\left(\left(\prod_{j=1}^{N-1}(\vartheta_x+b_j)\right) \partial_x- \left(\prod_{i=1}^{N}(\vartheta_x+a_i) \right)\right).
\end{gather}
We call \eqref{nPn-1} the generalized hypergeometric differential equation with the large parameter. The singular points of the equation \eqref{nFn-1eq} (or \eqref{NFN-1eq}) are $0$, $1$ and $\infty$.

The differential operator \eqref{nPn-1} is a WKB type operator defined on $\mathbb{C}$ in the sense of \cite[Definition~2.1]{AKKoT}.
Let $\xi$ denote the dual variable of~$x$. Then~$\xi$ is regarded as the symbol of the differential operator $\partial_{x}$ (cf.~\cite{A,AKKoT2, AKKoT}). We introduce a new variable $\zeta$ by
setting $\zeta=\eta^{-1}\xi$. Since the large parameter $\eta$ is regarded as the dual variable of the variable~$y$ of the Borel plane \cite{KT}, $\eta$ indicates the symbol of $\partial_{y}=\partial/\partial y$. Hence $\zeta$ designates the symbol of the microdifferential operator $\partial_{y}^{-1}\partial_{x}$.
Then the total symbol $\sigma({}_N P_{N-1})(x,\zeta,\eta)$ in the sense of \cite{AKKoT} can be considered as a~function $x$, $\zeta$ and $\eta^{-1}$.
We write $\sigma_k({}_N P_{N-1})(x,\zeta)$ the coefficient of $\eta^{-k}$ of $\sigma({}_N P_{N-1})(x,\zeta,\eta)$. Then the total symbol can be expressed in the form
\begin{gather}\label{2:(2.1)}
\sigma({}_N P_{N-1})(x,\zeta,\eta)=\sum_{k=0}^{N} \eta^{-k} \sigma_k({}_N P_{N-1})(x,\zeta).
\end{gather}
We call the leading term $\sigma_0({}_N P_{N-1})(x,\zeta)$ of~\eqref{2:(2.1)}
the principal symbol of ${}_N P_{N-1}$.
For an $m$-tuple $\bm{c}=(c_1,c_2,\dots,c_m)$ of parameters, we denote by $\mathfrak{s}_{\ell}(\bm{c})$ the elementary symmetric polynomial of degree $\ell$ of $c_1,c_2,\dots,c_m$. If $\ell> m$ or $\ell \leq-1$, we set $\mathfrak{s}_{\ell}(\bm{c})=0$.
\begin{Lemma}
The total symbol $\sigma({}_N P_{N-1})(x,\zeta,\eta)$ is written in the form{\rm:}
\begin{gather}
\sigma({}_N P_{N-1})(x,\zeta,\eta)\nonumber\\
\qquad{} =
\sum_{k=1}^N\left( \sum_{j=k}^N \!\left( \left\{{j-1\atop k-1}\right\}\mathfrak{s}_{N-j}({\bm b})-\left\{{j\atop k}\right\} \mathfrak{s}_{N-j}({\bm a})x \right)\!\right) x^{k-1}\eta^{-N+k} \zeta^k -\eta^{-N}\mathfrak{s}_N({\bm a}).\!\!\!\!\label{spPq}
\end{gather}
Here $\left\{{j\atop k}\right\}$ denotes the Stirling number of the second kind {\rm \cite{GKP,O}}.
\end{Lemma}

\begin{proof}By the definition of the Stirling number of the second kind, the $j$-th power of the Euler operator $\vartheta_x$ becomes
\begin{gather}\label{eulerst}
 \vartheta_x^j=\sum_{k=0}^j\left\{{j\atop k}\right\}x^k \partial_x^k, \qquad j\geq 0.
\end{gather}
Expanding the right-hand side of \eqref{nPn-1} and using \eqref{eulerst}, we have
\begin{align*}
\eta^{N} {}_NP_{N-1}
&=\sum_{\substack{0\leq j \leq N-1 \\ 0\leq k \leq j}} \left\{{j\atop k}\right\} \mathfrak{s}_{N-1-j}({\bm b}) x^k \partial_x^{k+1}-\sum_{\substack{0\leq j \leq N \\ 0\leq k \leq j}} \left\{{j\atop k}\right\} \mathfrak{s}_{N-j}({\bm a}) x^k \partial_x^k.
\end{align*}
Hence, we obtain~\eqref{spPq}.
\end{proof}

\begin{Remark}
The principal symbol $\sigma_0 ({}_N P_{N-1})(x,\zeta)$ is written in the form:
\begin{align}
\sigma_0 ({}_N P_{N-1})(x,\zeta)&=\zeta \prod_{j=1}^{N-1}(x \zeta+b_{j,1})- \prod_{i=1}^{N}(x \zeta+a_{i,1}) \notag\\
&=\sum_{k=1}^N \big(\mathfrak{s}_{N-k}({\bm b}_{1})-\mathfrak{s}_{N-k}({\bm a}_{1})x\big)x^{k-1}\zeta^{k}-\mathfrak{s}_N({\bm a}_{1}).\label{P0prod}
\end{align}
Here we set ${\bm a}_{1}=(a_{1,1},a_{2,1},\dots, a_{N,1})$ and ${\bm b}_{1}=(b_{1,1},b_{2,1},\dots, b_{N-1,1})$.
\end{Remark}

\subsection{Turning points and WKB solutions}

A point $x_\ast \in \mathbb{C}$ is called a turning point of ${}_N P_{N-1}$ with the characteristic value $\zeta_\ast$ if
\begin{gather*}
\sigma_0({}_N P_{N-1}) (x_{*}, \zeta_{*} )=\partial_{\zeta} \sigma_0({}_N P_{N-1}) (x_{*}, \zeta_{*} )=0
\end{gather*}
and $\sigma_0({}_N P_{N-1}) (x, \zeta )$ does not vanish identically as a function~$\zeta$
(see \cite[Definition~3.3]{AKKoT}).
The turning point $x_\ast$ with the characteristic value $\zeta_\ast$ is said to be simple if
\begin{gather*}
\partial_{x} \sigma_0({}_N P_{N-1}) (x_{*}, \zeta_{*} )\not=0,\qquad
\partial_{\zeta}^2 \sigma_0({}_N P_{N-1}) (x_{*}, \zeta_{*} )\not=0
\end{gather*}
(see \cite[Definition~3.6]{AKKoT}).

Let $\mathcal{P}({}_N P_{N-1})$ denote the set
\begin{gather*}
\{(x_\ast, \zeta_\ast) \,|\, \sigma_0({}_N P_{N-1}) (x_{*}, \zeta_{*} )=\partial_{\zeta} \sigma_0({}_N P_{N-1}) (x_{*}, \zeta_{*} )=0 \}.
\end{gather*}
This set can be regarded as a subset of $\mathbb{P}^2_{\mathbb{C}}$.
Let $\mathcal{P}_{\rm tp}({}_N P_{N-1})$ be the projection of $\mathcal{P}({}_N P_{N-1})$ to the $x$-space. Note that the singular points $0$, $1$, $\infty$ do not belong to $\mathcal{P}_{\rm tp}({}_N P_{N-1})$ for generic ${\bm a}_1$ and ${\bm b}_1$.
By the definition, the turning points of \eqref{nFn-1eq} should satisfy the following equation:
\begin{gather}\label{2:2.20}
\operatorname{res}_\zeta \left(\sigma_0({}_NP_{N-1})(x,\zeta), \partial_\zeta \sigma_0({}_NP_{N-1})(x,\zeta) \right)=0.
\end{gather}
Here $\operatorname{res}_\zeta (F, G)$ denotes the resultant of $F$ and $G$ with respect to $\zeta$.

\begin{Lemma}\label{slmm}
The resultant \eqref{2:2.20} is rewritten in the form
\begin{gather}
\operatorname{res}_\zeta \big(\sigma_0({}_NP_{N-1})(x,\zeta), \partial_\zeta \sigma_0({}_NP_{N-1})(x,\zeta) \big)\nonumber\\
\qquad{} =\frac{(-1)^{N-1}}{N^{N-2}} x^{(N-1)^2}(1-x)\operatorname{res}_\zeta \big(f(x,\zeta),g(x,\zeta)\big).\label{tpreseq}
\end{gather}
Here we set
\begin{gather}\label{tpf}
f(x,\zeta)=\sum_{k=0}^{N-1} (k+1) (\mathfrak{s}_{N-k-1}({\bm b}_1)-\mathfrak{s}_{N-k-1}({\bm a}_1)x) \zeta^k,\\
\label{tpg}
g(x,\zeta)=\sum_{k=1}^{N-1} (N-k) (\mathfrak{s}_{N-k}({\bm b}_1)-\mathfrak{s}_{N-k}({\bm a}_1)x) \zeta^k -N\mathfrak{s}_N({\bm a}_{1})x.
\end{gather}
\end{Lemma}

\begin{proof}
For the sake of simplicity, we show \eqref{tpreseq} for
the case $N=3$. General case can be proved in a similar way. For $N=3$, the principal symbol
is written in the form
\begin{gather*}
\sigma_0({}_3P_{2})(x,\zeta)=x^2q_3\zeta^3+x q_2\zeta^2+q_1\zeta+q_0,
\end{gather*}
Hence the left-hand side of \eqref{tpreseq} has the form
\begin{gather}\label{det1}
\left|
\begin{matrix}
q_3x^2& q_2x&q_1&q_0&0\\
0&q_3x^2& q_2x&q_1&q_0\\
3q_3x^2&2q_2 x&q_1&0&0\\
0&3q_3x^2&2q_2 x&q_1&0\\
0&0&3q_3x^2&2q_2 x&q_1\\
\end{matrix}
\right|.
\end{gather}
Firstly we eliminate the $(3,1)$-element by using the first row. Next we eliminate the
$(4,2)$-element by the second row. Then, expanding the determinant by the first column, we have
\begin{gather*}
(-1)^{2}q_3x^2
\left|
\begin{matrix}
q_3x^2& q_2x&q_1&q_0\\
q_2 x&2q_1&3q_0&0\\
0&q_2 x&2q_1&3q_0\\
0&3q_3x^2&2q_2 x&q_1\\
\end{matrix}
\right|.
\end{gather*}
We eliminate the $(1,4)$-element by using the third row and multiply the first row by~$3$. Then,
factoring~$x$ out from the first column, multiplying the third and the fourth columns by~$x$ and~$x^{2}$,
respectively, exchanging the second row and the fourth row and finally, exchanging the third row and the fourth row,
 we have
\begin{gather*}
\frac{(-1)^{2}}{3} q_3
\left|
\begin{matrix}
3q_3x& 2q_2x&q_1x&0\\
0&3q_3 x^{2}&2q_2x^{2}&q_1x^{2}\\
q_2 &2q_1&3q_0x&0\\
0&q_2x&2q_1 x&3q_0 x^{2}\\
\end{matrix}
\right|.
\end{gather*}
We factor $x$, $x^{2}$ and $x$ out from the first, the second and the fourth row, respectively.
Then we obtain the expression of~\eqref{det1} of the form
\begin{gather*}
\frac{(-1)^2}{3}q_3x^{4}
\left|
\begin{matrix}
3q_3&2q_2& q_1&0\\
0&3q_3&2q_2 &q_1\\
q_2 &2q_1&3q_0x&0\\
0&q_2 &2q_1&3q_0x\\
\end{matrix}
\right|.
\end{gather*}
Thus we have
\begin{gather*}
\operatorname{res}_\zeta(\sigma_0({}_3P_{2})(x,\zeta), \partial_\zeta \sigma_0({}_3P_{2})(x,\zeta) )=\frac{(-1)^2}{3}(1-x)x^{4}\operatorname{res}_\zeta \left(f(x,\zeta),g(x,\zeta) \right)
\end{gather*}
with
\begin{gather*}
f(x,\zeta)=3q_3\zeta^2+2q_2\zeta+ q_1,\\
g(x,\zeta)=q_2 \zeta^2+2q_1\zeta+3q_0x.
\end{gather*}
This proves Lemma~\ref{slmm} for $N=3$.
\end{proof}
We assume that the following conditions are satisfied.

\begin{asp}\label{asptp}\quad
\begin{itemize}\itemsep0pt
\item[(i)]
All turning points of the equation \eqref{nFn-1eq} are simple.
\item[(ii)] $\mathfrak{s}_{1}({\bm a}_1)\not=\mathfrak{s}_{1}({\bm b}_1)$.
\item[(iii)]
The leading coefficient of $\operatorname{res}_\zeta \left(f(x,\zeta),g(x,\zeta)\right)$ with respect to $x$ does not vanish.
\item[(iv)]
$
{\rm{Dis}}_x \left( \operatorname{res}_\zeta \left(f(x,\zeta),g(x,\zeta)\right)\right)\not=0.
$ Here ${\rm{Dis}}_x (F)$ denotes the discriminant of $F$ with respect to $x$.
\end{itemize}
\end{asp}
\begin{Lemma}\label{tpnum}
The number of elements in the set $\mathcal{P}({}_N P_{N-1})$ equals $2(N-1)$.
\end{Lemma}

\begin{proof}
From Lemma \ref{slmm}, a point $(x,\zeta)$ belongs to $\mathcal{P}({}_N P_{N-1})$ if and only if $(x,\zeta)$ satisfies $f(x,\zeta)=0$ and $g(x,\zeta)=0$. Here $f(x,\zeta)$ and $g(x,\zeta)$ denote \eqref{tpf} and \eqref{tpg}, respectively.
Eliminating $x$ from $f(x,\zeta)=0$ and $g(x,\zeta)=0$, we get the equation $h(\zeta)=0$ for $\zeta$, where
\begin{gather*}
h(\zeta)= \sum_{k=N}^{2(N-1)}\Bigg(\sum_{j=1}^{2N-k}
j(N-k-j+1)\big(\mathfrak{s}_{j-1}({\bm b}_1)\mathfrak{s}_{2N-k-j}({\bm a}_1)-\mathfrak{s}_{j-1}({\bm a}_1)\mathfrak{s}_{2N-k-j}({\bm b}_1)\big) \Bigg) \zeta^{k}\\
\hphantom{h(\zeta)=}{}
+\sum_{k=0}^{N-1} \Bigg(\sum_{j=0}^{k+1}
N(k+j+1)\big(\mathfrak{s}_{N-k+j-1}({\bm a}_1)\mathfrak{s}_{N-j}({\bm b}_1)-\mathfrak{s}_{N-k+j-1}({\bm b}_1)\mathfrak{s}_{N-j}({\bm a}_1)\big)
\Bigg) \zeta^{k}.
\end{gather*}
We can see that the coefficient of the highest degree part of $h(\zeta)$ does not vanish if $\mathfrak{s}_{1}({\bm a}_1)\not=\mathfrak{s}_{1}({\bm b}_1)$. The coefficient of~$\zeta^k$ of~$h(\zeta)$ is a homogeneous polynomial of degree $2N-k-1$ of $\bm{a}_1$ and $\bm{b}_1$ which does not vanish identically. Thus, we can show that there are distinct $2(N-1)$ roots of $h(\zeta)=0$ by using Assumption~\ref{asptp}.
\end{proof}

\begin{Remark}
If $N=3$,
\begin{gather*}
{\rm{Dis}}_x \left( \operatorname{res}_\zeta \left(f(x,\zeta),g(x,\zeta)\right)\right)=256 \prod_{i=1,2,3}a_{i,1}\prod_{\substack{i=1,2,3 \\ j=1,2}}(a_{i,1}-b_{j,1}) h'(\bm{a}_1,\bm{b}_1)^3
\end{gather*}
holds. Here $h'(\bm{a}_1,\bm{b}_1)$ is a homogeneous polynomial of degree $9$ with respect to $\bm{a}_1$, $\bm{b}_1$. From $h'(\bm{a}_1,\bm{b}_1)=0$, we have the following conditions:
\begin{itemize}\itemsep0pt
\item[(i)] $b_{j,1}\neq 0$ for any $j$, $1\leq j\leq 2$.
\item[(ii)] $b_{j,1} \neq b_{j',1}$ for $j\not=j'$, $1\leq j,j'\leq 2$.
\end{itemize}
The leading coefficient of $\operatorname{res}_\zeta (f(x,\zeta),g(x,\zeta))$ with respect to~$x$ does not vanish. Then we have
\begin{gather*}
\prod_{\substack{i\neq i'}}(a_{i,1}-a_{i',1})\neq0.
\end{gather*}
\end{Remark}

Outside the turning points, there are $N$ distinct roots of the algebraic equation
\[ \sigma_0({}_N P_{N-1})(x,\zeta)=0\] in $\zeta$ of degree $N$.
We call these roots the characteristic roots of ${}_{N}P_{N-1}$.
We consider the Laurent expansion at the singular point of each characteristic root. For the case of $\varrho=0$, we substitute $\zeta=\sum_{k=m}^{\infty} c_k x^{k}$ for \eqref{P0prod}. Taking note of the degree of the leading term of \eqref{P0prod} with respect to~$x$, we find $m=-1$ or $m=0$.
If $m=-1$, there are $N-1$ choices for the leading coefficient:
\begin{gather*}
c_{-1} = -b_{\ell,1}, \qquad \ell=1,2,\dots, N-1.
\end{gather*}
If $m=0$, we can see
\begin{gather*}
c_0=\frac{a_{1,1}a_{2,1} \cdots a_{N,1}}{b_{1,1}b_{2,1}\cdots b_{N-1,1}}.
\end{gather*}
The coefficients of higher-order terms $c_{m+\ell}$, $\ell>0$, are determined recursively.
Similarly, we can find the Laurent expansions of the characteristic roots at $x=\infty$.
\begin{Definition}\label{charexpand}
We take local numbering of the characteristic roots of ${}_{N}P_{N-1}$ as follows:
\begin{itemize}\itemsep0pt
\item[(i)]For $\varrho=0$, the roots are denoted by $\zeta_m^{(0)}$, $m=1,2,\dots,N$, so that
\begin{gather*}
\zeta_\ell^{(0)} =- \frac{b_{\ell,1}}{x}+O(1),\qquad \ell=1,2,\dots,N-1,\\
\zeta_N^{(0)} = \frac{a_{1,1}a_{2,1} \cdots a_{N,1}}{b_{1,1}\cdots b_{N-1,1}}+O(x)
\end{gather*}
hold.
\item[(ii)]For $\varrho=\infty$, the roots are denoted by $\zeta_m^{(\infty)}$, $m=1,2,\dots,N$, so that
\begin{gather*}
\zeta_m^{(\infty)} = -\frac{a_{m,1}}{x}+O\left(\frac{1}{x^2}\right)
\end{gather*}
hold.
\end{itemize}
\end{Definition}

Taking suitable branch cuts connecting simple turning points, we can regard this numbering is defined globally.
For a simple turning point~$x_\ast$ and a singular point $\varrho$, there are two numbers $j,k \in \{1,2,\dots, N\}$, $j\not=k$ such that $\zeta_j^{(\varrho)}(x_\ast)=\zeta_k^{(\varrho)}(x_\ast)$ holds. We consider mainly the case where $\varrho=0$. Then we say that $x_\ast$ is a simple turning point of type $(j,k)$ (see \cite[Definition~1.2.1]{HKT}).

The characteristic variety $\operatorname{Ch}({}_NP_{N-1})$ of ${}_NP_{N-1}$ is, by definition, an algebraic curve
\begin{gather*}
\operatorname{Ch}({}_NP_{N-1})=\{(x,\zeta) \,|\, \sigma_0({}_NP_{N-1})(x,\zeta)=0\}.
\end{gather*}
This can be regarded as a compact Riemann surface $\Sigma$. There is a natural projection
\begin{gather*}
\pi\colon \ \Sigma \rightarrow \mathbb{P}_{\mathbb{C}}^1,
\end{gather*}
which is an $N$-covering map.
\begin{Lemma}\label{genus}
The genus of $\Sigma$ equals $0$.
\end{Lemma}
\begin{proof}
By using the Riemann--Hurwitz formula \cite[Section~17]{OT}, we have the following relation:
\begin{gather*}
2-2g(\Sigma)=N\big(2-2g\big(\mathbb{P}^1_\mathbb{C}\big)\big)-\sum_{i=1}^{2(N-1)}(2-1).
\end{gather*}
Here $g(X)$ denotes the genus of the compact Riemann surface $X$. Hence we get
$g(\Sigma)=0$.
\end{proof}

A WKB solution $\psi$ of \eqref{nFn-1eq} is a formal solution of the form
\begin{gather*}
\psi=\exp \left(\int S \,{\rm d} x\right),\\
S=\eta S_{-1}+S_0+\eta^{-1}S_{1}+\cdots=\sum_{\ell=-1}^\infty \eta^{-\ell} S_\ell
\end{gather*}
(see \cite[Definition~3.2]{AKKoT}).
By inserting $\psi=\exp \big(\int S \,{\rm d}x \big)$ into \eqref{nFn-1eq},
we have the following nonlinear differential equation for~$S$:
\begin{gather}\label{RiS}
\operatorname{Ri}({}_N P_{N-1})(S)=0.
\end{gather}
Here we set
\begin{gather*}
\operatorname{Ri}({}_N P_{N-1})(S)=\exp\left(-\int S \,{\rm d}x\right){}_N P_{N-1}\exp\left(\int S \,{\rm d}x\right).
\end{gather*}
The leading term of the equation \eqref{RiS} with respect to $\eta^{-1}$ determines~$S_{-1}$.
Therefore $S_{-1}$ should satisfy $\sigma_0({}_NP_{N-1})(x,S_{-1})=0$.
Hence we can take one of the characteristic roots as~$S_{-1}$.
The higher-order terms $S_{j}$, $j=0,1,2, \dots$, can be determined recursively and uniquely outside the turning points if~$S_{-1}$ is chosen.
The following Lemma is obtained by the same way as \cite[Proposition~3.6]{OT}:
\begin{Lemma}\label{Sexpandreg}
Let $S^{(\varrho,m)}$, $m=1,2,\dots,N$, be the formal solutions of~\eqref{nFn-1eq} such that the numbering are consistent with that of the leading terms given in Definition~{\rm \ref{charexpand}}.
Then the formal solutions $S^{(\varrho,m)}$, $m=1,2,\dots,N$, have the following local behaviors near $x=\varrho$, $\varrho=0,\infty$:
\begin{gather}\label{Sexp01}
S^{(0,\ell)}=\frac{1-b_{\ell}}{x}+\frac{\prod\limits_{m=1}^{N} (1+a_{m}-b_{\ell})}{(b_{\ell}-2)\prod\limits_{\substack{m=1 \\ m\not=\ell}}^{N-1}(b_{\ell}-b_{m}-1)}+O(x),\qquad
\ell=1,2,\dots,N-1,\\
\label{Sexp02}
S^{(0,N)}=\frac{a_{1}a_{2} \cdots a_{N}}{b_{1}b_{2} \cdots b_{N-1}}+O(x),\\ \label{Sexp03}
S^{(\infty,\ell)}=-\frac{a_{\ell}}{x}- \frac{a_\ell \prod\limits_{m=1}^{N-1} (1+a_\ell-b_m)}{\prod\limits_{\substack{m=1 \\ m\not=\ell}}^{N} (1+a_\ell-a_m)}\frac{1}{x^2}+O\left( \frac{1}{x^3}\right), \qquad
\ell=1,2,\dots,N.
\end{gather}
Here we use the notation $O(x)$ in \eqref{Sexp01} and \eqref{Sexp02} in the sense that these parts can be written as
\begin{gather*}
x\sum_{j=-1}^\infty \eta^{-j} f_j(x)
\end{gather*}
with some holomorphic functions $f_j(x)$, $j=-1,0,1,\dots$, near $x=0$. The notation $O\big( {1}/{x^3}\big)$ in~\eqref{Sexp03} should be understood similarly.
\end{Lemma}

\subsection{Factorization}
\par We assume that there is a simple turning point $x_\ast$ of ${_N}P_{N-1}$ with the characteristic value
$\zeta_\ast$. Then $\sigma_0({}_NP_{N-1})(x, \zeta)$ is uniquely decomposed holomorphically in a neighborhood of $(x_\ast,\zeta_\ast)$ in the form
\begin{gather*}
\sigma_0({}_NP_{N-1})(x, \zeta)=l(x, \zeta) r(x, \zeta),
\end{gather*}
where $r(x,\zeta)$ is a Weierstrass polynomial of degree $2$ in $\zeta$ with the center at $(x_\ast, \zeta_\ast)$ and $l(x_\ast,\zeta_\ast) \allowbreak \not= 0$. By the definition, $r(x,\zeta)$ has the form
\begin{gather*}
r(x, \zeta)= (\zeta-\zeta_{*} )^{2}+f_{1}(x) (\zeta-\zeta_{*} )+f_{2}(x),
\end{gather*}
where $f_{j}(x)$ vanishes at $x=x_{*}$ for $j=1,2$. We may assume $\zeta_{j}^{(\varrho)}$ and $\zeta_{k}^{(\varrho)}$ are the roots of $r(x, \zeta)=0$. We set
\begin{gather*}
{}_N {\tilde P}_{N-1}= \frac{1}{x^{N-1}(1-x)} {}_N { P}_{N-1}.
\end{gather*}
Then $\sigma\big({}_N {\tilde P}_{N-1}\big)(x,\zeta,\eta)$ becomes a monic polynomial with respect to~$\zeta$:
\begin{gather*}
\sigma\big({}_N {\tilde P}_{N-1}\big)(x,\zeta,\eta)=\zeta^N+\sum_{k=0}^{N-1}\tilde p_k (x,\eta)\zeta^k.
\end{gather*}
It follows from \cite[Theorem~5.1]{AKKoT} that there uniquely exist differential operators $L$ and $R$ of
WKB type near $x_{*}$ which satisfy
\begin{gather}\label{fac}
{}_N {\tilde P}_{N-1}=LR
\end{gather}
and
\begin{itemize}\itemsep0pt
\item[(i)] The principal symbol $\sigma_0(R)(x, \zeta)$ of $R$ coincides with $r(x, \zeta)$.
\item[(ii)] For each $j>0$, the coefficient $\sigma_j(R)(x, \zeta)$ of $\eta^{-j}$ of the symbol of $R$ is of degree at most one in~$\zeta$.
\item[(iii)] The principal symbol $\sigma_{0}(L)(x, \zeta)$ of $L$ does not vanish at $(x_{*}, \zeta_{*})$.
\end{itemize}
Thus constructed operator $R$ has the form
\begin{gather*}
R=\big(\eta^{-1} \partial_{x}\big)^2+A(x, \eta) \eta^{-1} \partial_{x}+B(x, \eta)
\end{gather*}
with
\begin{gather*}
A(x, \eta)=A_{0}(x)+A_{1}(x) \eta^{-1}+A_{2}(x) \eta^{-2}+\cdots,\\
B(x, \eta)=B_{0}(x)+B_{1}(x) \eta^{-1}+B_{2}(x) \eta^{-2}+\cdots
\end{gather*}
and $L$ is an $(N-2)$-th order operator
\begin{gather*}
L=\big(\eta^{-1} \partial_{x}\big)^{N-2}+\sum_{k=0}^{N-3} L_k(x, \eta) \big(\eta^{-1} \partial_{x}\big)^k
\end{gather*}
with
\begin{gather*}
L_k(x, \eta)=L_{k,0}(x)+L_{k,1}(x)\eta^{-1}+L_{k,2}(x)\eta^{-2}+\cdots.
\end{gather*}
The operators $L$ and $R$ are constructed as follows.
Relation~\eqref{fac} yields
\begin{gather*}
\tilde p_2=A+L_0,\qquad
\tilde p_1=B+A L_0 +\eta^{-1} A',\qquad
\tilde p_0=B L_0+\eta^{-1} B'
\end{gather*}
for $N=3$ and
\begin{gather*}
\tilde p_{N-1}(x, \eta)=A+L_{N-3},\\
\tilde p_{k}(x, \eta)=\eta^{-N+k+2}\left( \binom{N-2}{k-1} \eta^{-1} A^{(N-k-1)}+\binom{N-2}{k}B^{(N-k-2)}\right)\\
\hphantom{\tilde p_{k}(x, \eta)=}{} + \sum_{j=k}^{N-3} \eta^{-j+k} \left( \binom{j}{k-1}\eta^{-1}A^{(j-k+1)}+\binom{j}{k}B^{(j-k)}\right) L_{j}\\
\hphantom{\tilde p_{k}(x, \eta)=}{}+ L_{k-2} +A L_{k-1},\qquad 2\leq k\leq N-2,
\\
\tilde p_{1}(x, \eta) = \eta^{-N+3}\big(\eta^{-1}A^{(N-2)} +(N-2)B^{(N-3)}\big)\\
\hphantom{\tilde p_{1}(x, \eta) =}{} + \sum_{j=1}^{N-3} \eta^{-j+1} \big( \eta^{-1} A^{(j)}+ j B^{(j-1)}\big)L_j+A L_0,\\
\tilde p_{0}(x, \eta)= \eta^{-N+2} B^{(N-2)}+\sum_{j=0}^{N-3} \eta^{-j}B^{(j)} L_{j}
\end{gather*}
for $N\geq 4$. These relations determine $A$, $B$ and $L_k$, $0\leq k \leq N-3$, if we choose the leading terms of them, which solve a system of algebraic equations. For example, if $N=3$, we obtain~$A$ by solving the following equation:
\begin{gather*}
A^3-2 \tilde p_2A^2 -\big(3\eta^{-1}A'-\eta^{-1}\tilde p_2-\tilde p_1-\tilde p_2^2\big)A \\
\qquad{}+2 \eta^{-1} \tilde p_2A' + \eta^{-1}A''+\tilde p_0 -\tilde p_1 \tilde p_2-\eta^{-1}\tilde p_1'=0.
\end{gather*}
We consider a WKB solution
\begin{gather*}
\phi =\exp \left( \int T \,{\rm d} x\right),\\
T =\eta T_{-1}+T_0+\eta^{-1}T_{1}+\cdots=\sum_{\ell=-1}^\infty \eta^{-\ell} T_\ell
\end{gather*}
of $R \phi=0$. Here $T$ is determined by the Riccati equation associated with this equation:
\begin{gather}\label{RiReq}
\operatorname{Ri}(R)(T)=\eta^{-2}\big(T^{2}+ T'\big)+A \eta^{-1} T+B=0.
\end{gather}
Since we have \eqref{fac}, $\phi$ is a formal solution of~\eqref{fac}. The following lemma shows that
we may regard it as a WKB solution of~\eqref{fac}.

\begin{Lemma}\label{decLR}
We have the following relation for $N\geq 3${\rm:}
\begin{gather*}
\operatorname{Ri}\big({}_N {\tilde P }_{N-1}\big)(S) = \sum_{k=0}^{N-2} \frac{\eta^{-k}}{k!}\operatorname{Ri}\big(\wick{\partial^k_\zeta \sigma(L)(x,\zeta,\eta)}\big)(S)\partial^k_x \operatorname{Ri}(R)(S).
\end{gather*}
Here $\wick{\sigma(P)(x,\zeta,\eta)}$ designates
the differential operator defined by the total symbol
$\sigma(P)(x,\zeta,\eta)$ $($see {\rm \cite[Definition 4.6]{A}} and
{\rm \cite{AKKoT})}. Hence $\operatorname{Ri}(R)(T)=0$ implies
$\operatorname{Ri}\big({}_N {\tilde P }_{N-1}\big)(T)=0$.
\end{Lemma}
\begin{Remark}
The differential operator $\wick{\partial^k_\zeta \sigma(L)(x,\zeta,\eta)}$ is recovered from
the total symbol $\partial^k_\zeta \sigma(L)(x,\zeta,\eta)$ by replacing $\zeta$ by
$\eta^{-1}\partial_{x}$ after moving all powers of $\zeta$ in each term in the symbol to the rightmost part.
\end{Remark}
\begin{proof}
Since we have \eqref{fac}, $\operatorname{Ri}\big({}_{N}\tilde P_{N-1}\big)(S)$ is computed
in the following form
\begin{align*}
\operatorname{Ri}\big({}_{N}\tilde P_{N-1}\big)(S)&={\rm e}^{-\int S\,{\rm d}x}LR {\rm e}^{\int S\,{\rm d}x}\\
&={\rm e}^{-\int S\,{\rm d}x}L{\rm e}^{\int S\,{\rm d}x}{\rm e}^{-\int S\,{\rm d}x}R {\rm e}^{\int S\,{\rm d}x}.
\end{align*}
By the definition of $\operatorname{Ri}(R)(S)$, this equals
\[
{\rm e}^{-\int S\,{\rm d}x}L{\rm e}^{\int S\,{\rm d}x}\operatorname{Ri}(R)(S),
\]
which can be written as
\[
{\rm e}^{-\int S\,{\rm d}x}L\operatorname{Ri}(R)(S){\rm e}^{\int S\,{\rm d}x}.
\]
The product of the differential operator $L$ of order $N-2$ and the multiplying operator $\operatorname{Ri}(R)(S)$ has the symbol
\[
\sum_{k=0}^{N-2}\frac{\eta^{-k}}{k!}\partial_{\zeta}^{k}\sigma(L)(x,\zeta,\eta)\partial_{x}^{k}\operatorname{Ri}(R)(S).
\]
Since $\partial_{x}^{k}\operatorname{Ri}(R)(S)$ does not contain $\zeta$ for each $k$, we can write
\begin{align*}
L\operatorname{Ri}(R)(S)&=\wick{ \sum_{k=0}^{N-2}\frac{\eta^{-k}}{k!}\partial_{\zeta}^{k}\sigma(L)(x,\zeta,\eta)\partial_{x}^{k}\operatorname{Ri}(R)(S)}\\
&= \sum_{k=0}^{N-2}\frac{\eta^{-k}}{k!}\partial_{x}^{k}\operatorname{Ri}(R)(S) \wick{\partial_{\zeta}^{k}\sigma(L)(x,\zeta,\eta)}.
\end{align*}
Thus we have
\[
{\rm e}^{-\int S\,{\rm d}x}L\operatorname{Ri}(R)(S){\rm e}^{\int S\,{\rm d}x}=\sum_{k=0}^{N-2}\frac{\eta^{-k}}{k!}\operatorname{Ri}\big(\wick{\partial_{\zeta}^{k}\sigma(L)(x,\zeta,\eta)}\big)(S)\partial_{x}^{k}\operatorname{Ri}(R)(S).
\]
This proves the lemma.
\end{proof}

\subsection{Definition of Voros coefficients}
Recall that the characteristic roots $\zeta_{j}^{(\varrho)}$ and $\zeta_{k}^{(\varrho)}$ introduced in the preceding section satisfy
the quadratic equation
\begin{gather*}
\zeta^{2}+A_{0} \zeta+B_{0}=0.
\end{gather*}
We may assume
\begin{gather*}
S_{-1}^{(\varrho,j)}=\zeta_{j}=\frac{-A_{0}+\sqrt{A_{0}^{2}-4 B_{0}}}{2}, \qquad S_{-1}^{(\varrho,k)}=\zeta_{k}=\frac{-A_{0}-\sqrt{A_{0}^{2}-4 B_{0}}}{2}
\end{gather*}
for suitable choice of the branch of the square root. We set
\begin{gather}\label{Sevenodd}
S_{{\rm even}}^{(\varrho,j,k)}=\frac{1}{2}\big(S^{(\varrho,j)}+S^{(\varrho,k)}\big), \qquad S_{{\rm odd}}^{(\varrho,j,k)}=\frac{1}{2}\big(S^{(\varrho,j)}-S^{(\varrho,k)}\big).
\end{gather}
Then we have
\begin{gather*}
S_{{\rm even}}^{(\varrho,j, k)}+\frac{1}{2} \eta A=-\frac{1}{2} \frac{\partial_x S_{{\rm odd}}^{(\varrho,j, k)}}{S_{{\rm odd}}^{(\varrho,j, k)}}=-\frac{1}{2} \frac{d}{d x} \log S_{{\rm odd}}^{(\varrho,j, k)}
\end{gather*}
by using \eqref{RiReq}.
Thus we can take the following normalization of integration:
\begin{gather*}
\int S_{{\rm even}}^{(\varrho,j, k)}\, {\rm d}x=-\frac{1}{2} \log S_{{\rm odd}}^{(\varrho,j, k)}-\frac{1}{2} \eta \int A \,{\rm d}x.
\end{gather*}
Here we choose a primitive function $\int A \,{\rm d}x$.
\begin{Definition}
The WKB solutions $\psi_{\pm}^{(\tau)}$ of \eqref{nFn-1eq} normalized
at the turning point $\tau$ of type $(j,k)$ are defined by
\begin{gather*}
\psi_{\pm}^{(\tau)} =\frac{1}{\sqrt{S_{{\rm odd}}^{(\varrho,j,k)}}} \exp \left(-\frac{1}{2} \eta \int A \,{\rm d}x\right) \exp \left(\pm\int_{\tau}^{x} S_{{\rm odd}}^{(\varrho,j,k)} \,{\rm d} x\right).
\end{gather*}
Here the integration of $S_{{\rm odd}}^{(\varrho,j,k)}$ from $\tau$ to $x$ is understood as a half of the contour integral of
it on the path starting from $x$ on the $k$-th sheet, going around $\tau$ counterclockwise and back to~$x$ on the $j$-th
sheet.
\end{Definition}

We denote by $S^{(\varrho,j,k)}_{{\rm odd},\ell}$ the coefficient of $\eta^{-\ell}$ of $S^{(\varrho,j,k)}_{{\rm odd}}$ and set $
S^{(\varrho,j,k)}_{{\rm odd},\leq 0}= \eta S^{(\varrho,j,k)}_{{\rm odd},-1}+S^{(\varrho,j,k)}_{{\rm odd},0}$.
It follows from Lemma~\ref{Sexpandreg} and \eqref{Sevenodd} that
\begin{gather*}
\underset{x=\varrho}{\operatorname{Res}}\, S_{\rm odd}^{(\varrho,j,k)} \,{\rm d} x=\underset{x=\varrho}{\operatorname{Res}} \, S_{{\rm odd}, \leq 0}^{(\varrho,j,k)} \,{\rm d}x
\end{gather*}
holds.

\begin{Definition}
The WKB solutions $\psi_{\pm}^{(\varrho)}$ of \eqref{nFn-1eq} normalized
at the singular point $\varrho$ are defined by
\begin{gather*}
\psi_{\pm}^{(\varrho)} =\frac{1}{\sqrt{S^{(\varrho,j,k)}_{\rm odd}}}\exp \left(-\frac{1}{2}\eta \int\! A\,{\rm d}x \right)
 \exp \left(\pm\int^x_{\varrho}\! \left(S^{(\varrho,j,k)}_{\rm odd}-S^{(\varrho,j,k)}_{{\rm odd},\leq 0}\right)\,{\rm d}x \pm \int^x_{\tau} \! S^{(\varrho,j,k)}_{{\rm odd},\leq 0}\,{\rm d}x\right).
\end{gather*}
\end{Definition}

\begin{Definition}The Voros coefficient $V_{\varrho}^{(j, k)}$ at $x=\varrho$ of type $(j, k)$ is defined by
\begin{gather*}
V^{(j,k)}_\varrho=\int_\varrho^\tau \left(S^{(\varrho,j,k)}_{\rm{odd}}-S^{(\varrho,j,k)}_{\rm{odd},\leq 0}\right)\,{\rm d}x.
\end{gather*}
\end{Definition}
The Voros coefficient relates two kinds of the normalization of WKB solutions defined in the preceding subsection. We have the following formal relations:
\begin{gather*}
\psi_{\pm}^{(\varrho)}=\exp\big( V_{\varrho}^{(j, k)}\big) \psi_{\pm}^{(\tau)}.
\end{gather*}
The Voros coefficient is rewritten as
\begin{gather*}
V^{(j,k)}_\varrho=\frac{1}{2} \int_{\gamma_{j,k}} (S-\eta S_{-1}-S_{0} )\,{\rm d}x.
\end{gather*}
Here we consider that $S$ is defined on the Riemann surface $\Sigma$ of $S_{-1}$ and $\gamma_{j,k}$ is a path on $\Sigma$ starting from the singular point $\varrho$ on the $j$-th sheet, going to and detouring $\tau$ counterclockwise on the base space and back to the singular point $\varrho$ on the $k$-th sheet.
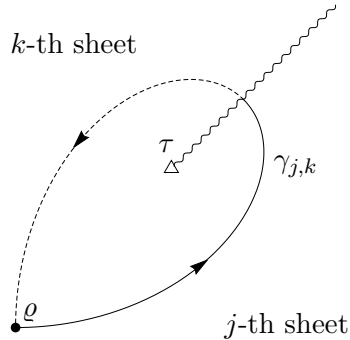
\begin{figure}[h]\centering
\begin{overpic}[scale=0.50]{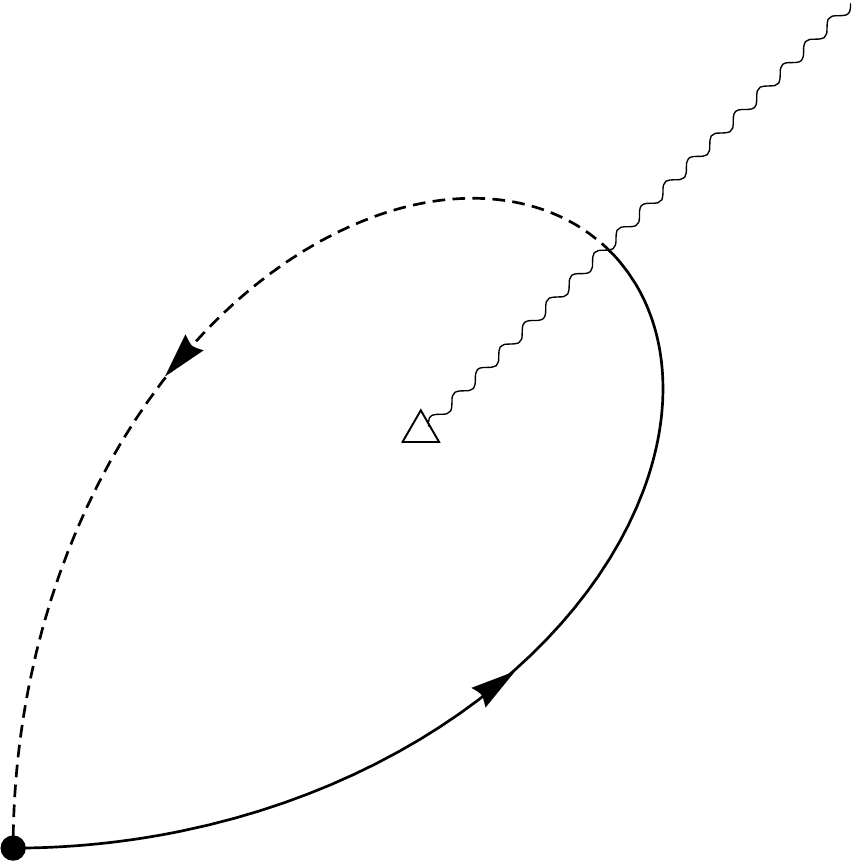}
\put(0,85){$k$-th sheet}
\put(65,0){$j$-th sheet}
\put(45,55){$\tau$}
\put(3.5,5){$\varrho$}
\put(80,50){$\gamma_{j,k}$}
\end{overpic}

\caption{The path $\gamma_{j,k}$.}
\end{figure}

\begin{Remark}
For any $1\leq j,k,\ell \leq N$, the following relation holds from Lemma \ref{genus}:
\begin{gather*}
V^{(j,k)}_\varrho+V^{(k,\ell)}_\varrho=V^{(j,\ell)}_\varrho.
\end{gather*}
Especially, $V^{(j,k)}_\varrho=-V^{(k,j)}_\varrho$ holds.
\end{Remark}

\section{Explicit forms of the Voros coefficients}\label{section3}
\subsection{Ladder operators}
\begin{Lemma}
Let $\mathcal{S} \left(\begin{smallmatrix} {\bm a}\\ {\bm b}\end{smallmatrix}\right)$ be the linear space of all solutions of ${}_NP_{N-1} \psi=0$. The operators
\begin{gather*}
H(a_i)=\vartheta_x+a_i, \qquad i=1,2,\dots,N,\\
B(b_j)=\vartheta_x+b_j,\qquad j=1,2,\dots,N-1
\end{gather*}
induce homomorphisms
\begin{gather} \notag
H(a_i)\colon \ \mathcal{S} \left(\begin{matrix}{\bm{a}} \\ {\bm{b}}\end{matrix}\right) \rightarrow \mathcal{S} \left(\begin{matrix}{\bm{a}}+{\bf{e}}_i \\ {\bm{b}}\end{matrix}\right), \qquad i=1,2,\dots,N,\\
B(b_j)\colon \ \mathcal{S} \left(\begin{matrix}{\bm{a}} \\ {\bm{b}}+{\bf{e}}_j \end{matrix}\right) \rightarrow \mathcal{S} \left(\begin{matrix}{\bm{a}}\\ {\bm{b}}\end{matrix}\right), \qquad j=1,2,\dots,N-1,\label{LoB}
\end{gather}
respectively. Here ${\bf e}_m$ denotes the $m$-th unit vector.
\end{Lemma}
\begin{proof}
We set
\begin{gather}\label{hatP}
{}_N\hat P_{N-1} ={}_N P_{N-1} \big|_{ a_{i,1} \mapsto a_{i,1}+\eta^{-1} }
\end{gather}
for $1\leq i \leq N$. Then we have the following relation for $1\leq i \leq N$:
\begin{gather}\label{eq3.5}
H(a_i) x {}_{N}P_{N-1}=x {}_{N}\hat P_{N-1} H(a_i).
\end{gather}
Hence
$\psi \in \mathcal{S} \left(\begin{smallmatrix}{\bm{a}} \\ {\bm{b}}\end{smallmatrix}\right)$ implies
$H(a_i) \psi \in \mathcal{S}\left(\begin{smallmatrix}{\bm{a}}+{\bf e}_i \\ {\bm{b}}\end{smallmatrix}\right)$. Similarly, \eqref{LoB} can be proved.
\end{proof}
\begin{Lemma}\label{RiSS}
Let $S$ be a formal solution of $\operatorname{Ri}({}_NP_{N-1})(S)=0$. Then
\begin{gather}\label{RiSSeqa}
\operatorname{Ri}\big({}_N\hat P_{N-1}\big)\big(\partial_x \log (x S +a_i) +S\big)=0,\qquad i=1,2,\dots,N,\\ \label{RiSSeqb}
\operatorname{Ri}({}_N P_{N-1})\big(\partial_x \log \big(x \hat S +b_j\big) + \hat S\big)=0, \qquad j=1,2,\dots,N-1
\end{gather}
hold. Here we set \eqref{hatP} and
\begin{gather}\label{hatS}
\hat S = S \big|_{ b_{j,1} \mapsto b_{j,1}+\eta^{-1} }
\end{gather}
for $j=1,2,\dots, N-1$.
\end{Lemma}

\begin{proof}
We fix the index $i$, $1\leq i\leq N$. By the definition of $\operatorname{Ri}(\ast)(\star)$, we have the following relations:
\begin{gather}\label{ritemp11}
\operatorname{Ri}\big({}_N\hat P_{N-1}\big)\big(\partial_x \log(xS+a_i)+S\big)=\frac{1}{xS+a_i}{\rm e}^{-\int S\,{\rm d}x} {}_N\hat P_{N-1} (xS+a_i){\rm e}^{\int S\,{\rm d}x},\\ \label{ritemp12}
\operatorname{Ri}\big({}_N \hat P_{N-1} H(a_i)\big)(S)={\rm e}^{-\int S\,{\rm d}x} {}_N\hat P_{N-1} (xS+a_i){\rm e}^{\int S\,{\rm d}x}.
\end{gather}
Combining \eqref{ritemp11} and \eqref{ritemp12}, we have
\begin{gather}\label{ritemp2}
\operatorname{Ri}\big({}_N\hat P_{N-1}\big)\big(\partial_x \log(xS+a_i)+S\big)=\frac{1}{xS+a_i}
\operatorname{Ri}\big({}_N \hat P_{N-1} H(a_i)\big)(S).
\end{gather}
From \eqref{eq3.5}, we get
\begin{gather}\label{ritemp3}
\operatorname{Ri}(H(a_i) x {}_N P_{N-1})(S)=\operatorname{Ri}\big(x {}_N \hat P_{N-1} H(a_i)\big)(S).
\end{gather}
By using a similar argument employed in the proof of Lemma \ref{decLR}, the left-hand side of \eqref{ritemp3} can be written in the form:
\begin{gather}\label{ritemp1}
\operatorname{Ri}(H(a_i) x {}_N P_{N-1})(S)= \sum_{k\geq 0} \frac{\eta^{-k}}{k!}\operatorname{Ri}\big(\wick{\partial^k_\zeta (x \eta \zeta+a_i)}\big)(S)\partial^k_x \operatorname{Ri}(x {}_N P_{N-1})(S).
\end{gather}
Then the right-hand side of \eqref{ritemp1} is equal to
\begin{gather*}
(xS+a_i)x \operatorname{Ri}({}_N P_{N-1})(S)+x\big(\operatorname{Ri}({}_N P_{N-1})(S)+x \partial_x \operatorname{Ri}({}_N P_{N-1})(S)\big).
\end{gather*}
Here we used the relation
\begin{gather*}
 \operatorname{Ri}(x {}_N P_{N-1})(S)=x\operatorname{Ri}({}_N P_{N-1})(S).
\end{gather*}
Since $\operatorname{Ri}({}_N P_{N-1})(S)=0$, the right-hand side of \eqref{ritemp1} vanishes. Combining \eqref{ritemp2}, \eqref{ritemp3} and~\eqref{ritemp1}, we obtain~\eqref{RiSSeqa}. The relation~\eqref{RiSSeqb} can be proved as well.
\end{proof}

\begin{Lemma}\label{sd}
The formal solution $S$ satisfies the following relations:
\begin{gather*}
\Delta_{a_{i,1}}S =\partial_x \log ( x S +a_i), \qquad i=1,2,\dots,N,\\
\Delta_{b_{j,1}}S =-\partial_x \log \big( x \hat S +b_j\big),\qquad j=1,2,\dots,N-1,
\end{gather*}
where we set \eqref{hatS} and
\begin{gather*}
\Delta_{\rho} S=S\big|_{ \rho\, \mapsto \rho +\eta^{-1} }-S
\end{gather*}
for $\rho= a_{1,1}, a_{2,1},\dots, a_{N,1}, b_{1,1},b_{2,1},\dots, b_{N-1,1}$.
\end{Lemma}
\begin{Remark}
Note that we distinguish $\rho$ from $\varrho$.
\end{Remark}

\begin{Remark}
The difference operator $\Delta_{\rho}$ in $\rho$ by $\eta^{-1}$can be written in the following form by using a formal differential operator of infinite order:
\begin{gather*}
\Delta_{\rho} ={\rm e}^{\eta^{-1}\partial_\rho }-1.
\end{gather*}
\end{Remark}
\begin{proof}
We denote $\hat{\psi}= H({a_i}) \psi$, which belongs to $\mathcal{S} \left(\begin{smallmatrix}{\bm a}+{\bf e}_i\\ {\bm b}\end{smallmatrix}\right)$. Then
\begin{gather*}
\partial_x \log \hat{\psi} = \partial_x\log (xS+a_i)+S
\end{gather*}
holds. It follows from Lemma \ref{RiSS} that $\hat{S}$ satisfies the equation obtained from $\operatorname{Ri}({}_NP_{N-1})(S)=0$ by replacing $a_i$ by
$a_i+1$. Hence we have
\begin{gather}
\partial_x\log (xS+a_i)+S =\frac{(\eta S_{-1}+S_0+\cdots)+\partial_x(\eta S_{-1}+S_0+\cdots)}{x(\eta S_{-1}+S_0+\cdots)+a_i}\nonumber\\
\hphantom{\partial_x\log (xS+a_i)+S =}{} +(\eta S_{-1}+S_0+\cdots).\label{tempa1}
\end{gather}
The leading term with respect to $\eta^{-1}$ of the right-hand side of \eqref{tempa1} equals $S_{-1}$.
We can choose the leading term of $\hat{S}$ as $S_{-1}\big|_{a_{i,1}\mapsto a_{i,1}+\eta^{-1} }$.
Thus, we conclude that $\hat{S}=S\big|_{a_{i,1}\mapsto a_{i,1}+\eta^{-1} }$.
Similarly, Lemma~\ref{sd} can be proved for $b_{1,1}, b_{2,1},\dots, b_{N-1,1}$.
\end{proof}

\begin{Lemma}\label{ddeqlem}
Let $(\rho,\rho_0)$ denote one of $(a_{1,1}, a_{1,0}),\dots ,(b_{N-1,1}, b_{N-1,0})$. The Voros coefficient $V_\varrho^{(j,k)}$, $j<k$, satisfies the following differential-difference equations:
\begin{gather}\label{ddeq}
\partial_{\rho} \Delta_{\rho}V_{\varrho}^{(j,k)}= f^{(j,k)}_{\rho}+\Delta_\rho g^{(j,k)}_\rho.
\end{gather}
Here $f^{(j,k)}_{\rho}$ are given as follows and here $g^{(j,k)}_\rho$ is a linear function of~$\eta$:
\begin{itemize}\itemsep0pt
\item[{\rm{(i)}}] If $\varrho=0$,
\begin{gather*}
2f_{a_{i,1}}^{(j,k)} =\frac{\eta}{1+a_i-b_k}-\frac{\eta}{1+a_i-b_j}, \qquad k\not=N,\\
2f_{a_{i,1}}^{(j,N)}=\frac{\eta}{a_i}-\frac{\eta}{1+a_i-b_j},\\
2f_{b_{m,1}}^{(j,k)}=\frac{\eta}{1+b_m-b_j}-\frac{\eta}{1+b_m-b_k},\qquad k\not=N, m\not=j,k,\\
2f_{b_{j,1}}^{(j,k)}= -\frac{\eta}{1+b_j-b_k}-\frac{\eta}{b_j-1}-\sum_{\substack{\ell=1 \\ \ell \not=j}}^{N-1} \frac{\eta}{b_{j}-b_{\ell}}-\sum_{\ell=1}^N \frac{\eta}{a_\ell-b_j},\qquad k\not=N,\\
2f_{b_{k,1}}^{(j,k)}=\frac{\eta}{1+b_k-b_j}+\frac{\eta}{b_k-1}+\sum_{\substack{\ell=1 \\ \ell \not=k}}^{N-1} \frac{\eta}{b_{k}-b_{\ell}}+\sum_{\ell=1}^N \frac{\eta}{a_{\ell}-b_k},\qquad k\not=N,\\
2f_{b_{m,1}}^{(j,N)}=\frac{\eta}{1+b_m-b_j}-\frac{\eta}{b_m}, \qquad m\not= j, \\
2f_{b_{j,1}}^{(j,N)}=\sum_{\ell=1}^N\frac{\eta}{b_j-a_\ell}-\frac{\eta}{b_j}-\frac{\eta}{b_j-1}-\sum_{\substack{\ell=1 \\ \ell \not=j}}^{N-1}\frac{\eta}{b_j-b_\ell},
\end{gather*}
$i=1,2,\dots,N$ and $m=1,2,\dots, N-1$.
\item[{\rm{(ii)}}] If $\varrho=\infty$,
\begin{gather*}
2f_{a_{i,1}}^{(j,k)} =\frac{\eta}{a_i-a_k}-\frac{\eta}{a_i-a_j},\qquad i\not=j,k,\\
2f_{a_{j,1}}^{(j,k)}=-\frac{\eta }{a_k-a_j}-\frac{\eta }{a_j}+\sum_{\substack{\ell=1 \\ \ell \not=j}}^{N} \frac{\eta }{1+a_j-a_\ell}-\sum_{\ell=1}^{N-1} \frac{\eta }{1+a_j-b_\ell},
\\
2f_{a_{k,1}}^{(j,k)}=\frac{\eta }{a_j-a_k}+\frac{\eta }{a_k}-\sum_{\substack{\ell=1 \\ \ell \not=k}}^{N} \frac{\eta }{1+a_k-a_\ell}+\sum_{\ell=1}^{N-1} \frac{\eta }{1+a_k-b_\ell},
\\
2f_{b_{m,1}}^{(j,k)}=\frac{\eta}{b_m-a_j}-\frac{\eta}{b_m-a_k},
\end{gather*}
$i=1,2,\dots,N$ and $m=1,2,\dots, N-1$.
\end{itemize}
\end{Lemma}
The system of differential-difference equations \eqref{ddeq} satisfy the compatibility conditions:
\begin{gather*}
\partial_{\tau} \Delta_{\tau}\partial_{\rho} \Delta_{\rho}V_\varrho^{(j,k)}=\partial_{\rho} \Delta_{\rho}\partial_{\tau} \Delta_{\tau}V_\varrho^{(j,k)},\qquad \forall \rho, \tau \in \{ a_{1,1}, \dots ,a_{N,1}, b_{1,1}, \dots,b_{N-1,1} \}.
\end{gather*}

\begin{proof}
We consider the integral
\begin{gather*}
\int_{{\tilde{\gamma}}_{j,k}} \Delta_\rho S\,{\rm d}x
\end{gather*}
for $ \rho \in \{ a_{1,1}, \dots ,a_{N,1}, b_{1,1}, \dots,b_{N-1,1} \}$.
Here the path $\tilde \gamma_{j,k}$ of integration is taken as follows: Let~$x$ be a point near the origin and $x^{(m)}$ the point on the $m$-th sheet of Riemann surface $\Sigma$ such that $\pi(x^{(m)})=x$. We take a path on $\Sigma$ starting from~$x^{(j)}$, going to and detouring $\tau$ counterclockwise on the base space and back to $x^{(k)}$ and denote this by $\tilde \gamma_{j,k}$.
\begin{figure}[h]\centering
\begin{overpic}[scale=0.50]{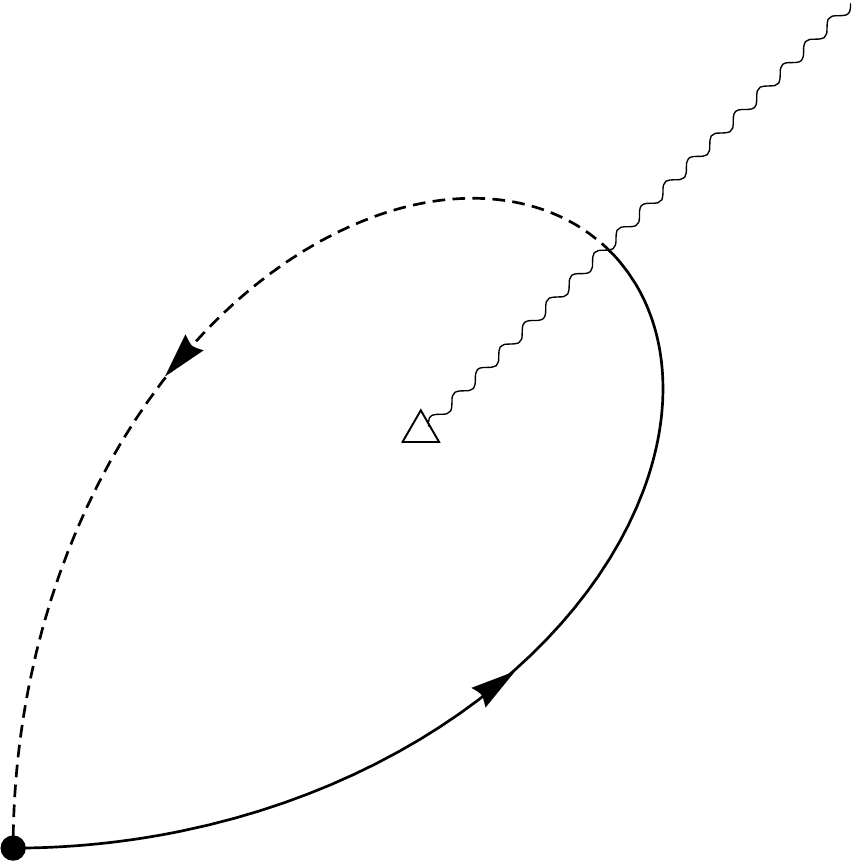}
\put(0,85){$k$-th sheet}
\put(65,0){$j$-th sheet}
\put(45,55){$\tau$}
\put(3.5,5){$x$}
\put(80,50){$\tilde \gamma_{j,k}$}
\end{overpic}

\caption{The path $\tilde \gamma_{j,k}$.}
\end{figure}
It follows from Lemma \ref{Sexpandreg} that this integral can be written in the form
\begin{gather*}
c^{(j,k)}_{-1} \log x+c_0^{(j,k)}+O(x)
\end{gather*}
with some constant series $c^{(j,k)}_{-1}$ and $c^{(j,k)}_{0}$ of $\eta^{-1}$. Here we use the notation $O(x)$ in the same sense as in Lemma \ref{Sexpandreg}.
Let $d_{m}^{(j,k)}$ denote the coefficient of $\eta^{-m}$ of $c_0^{(j,k)}$:
\begin{gather}
c_0^{(j,k)}=\sum _{m=0}^\infty d_{m}^{(j,k)} \eta^{-m}.\label{ddef}
\end{gather}
On the other hand, by using Lemma~\ref{sd}, we have
\begin{gather}\label{DetlaSint}
\int_{{\tilde{\gamma}}_{j,k}} \Delta_\rho S\,{\rm d}x=
\begin{cases}
 \displaystyle \log \left( \frac{x S^{(0,k)}+a_i}{x S^{(0,j)}+a_i} \right),& \rho=a_{i,1}, \ 1\leq i \leq N, \\
 \displaystyle -\log \left( \frac{x \hat S^{(0,k)}+b_m}{x \hat S^{(0,j)}+b_m} \right), & \rho=b_{m,1}, \ 1\leq m \leq N-1.
 \end{cases}
\end{gather}
Expanding the right-hand side of~\eqref{DetlaSint} with respect to~$x$, we obtain the coefficient $c_0^{(j,k)}$.
Moreover, Lemma~\ref{Sexpandreg} implies that $c_0^{(j,k)}$ is written as the logarithm of a product of several linear functions in~$\rho$, whose coefficients of~$\rho$ are~$-1$ or~$1$.

Let $F_{\ell}^{(m)}$ be a primitive function of $S^{(0,m)}_{\ell}$ with respect to $x$ and let $\int^{x} {\rm d}x$ denote the termwise integration in~$x$ of the Laurent expansion of the integrand at $x=0$, i.e.,
\begin{gather*}
\int^{x} \left(\frac{p_{-1}}{x}+p_0+p_1x+\cdots\right) {\rm d}x= p_{-1}\log x+p_0 x +\frac{1}{2}p_1x^2+\cdots.
\end{gather*}
We set
\begin{gather}\label{koike}
r^{(j,k)}_\ell=F^{(k)}_{\ell}(x)-F^{(j)}_{\ell}(x)-\left(\int^{x} S^{(0,k)}_{\ell}\,{\rm d}x -\int^{x} S^{(0,j)}_{\ell} \,{\rm d}x \right).
\end{gather}
Note that $r^{(j,k)}_\ell$ does not depend on $x$ and $\eta$.
We let apply $\Delta_\rho$ on both sides of~\eqref{koike}. Then we have
\begin{gather}\label{tempdeltasr}
\int_{\tilde\gamma_{j,k}} \Delta_\rho S_{\ell} \,{\rm d}x= \int^{x} \Delta_\rho S^{(0,k)}_{\ell}\,{\rm d}x - \int^{x} \Delta_\rho S^{(0,j)}_{\ell}\,{\rm d}x + \Delta_\rho r_\ell^{(j,k)}.
\end{gather}
Since $\Delta_\rho={\rm e}^{\eta^{-1}\partial_\rho}-1$, the right most term of~\eqref{tempdeltasr} can be written in the form
\begin{gather*}
\Delta_\rho r_\ell^{(j,k)} =\left(\eta^{-1}\partial_{\rho} +\frac{1}{2} \eta^{-2}\partial_{\rho}^2+\cdots \right) r_\ell^{(j,k)}.
\end{gather*}
Multiplying both members of~\eqref{tempdeltasr} by $\eta^{-\ell}$ and summing them up in~$\ell$, we have
\begin{gather}\label{427}
\int_{\tilde\gamma_{j,k}} \Delta_\rho S\,{\rm d}x = \sum_{\ell=-1}^\infty \eta^{-\ell} \left(\int^{x} \Delta_\rho S^{(0,k)}_{\ell}\,{\rm d}x - \int^{x} \Delta_\rho S^{(0,j)}_{\ell}\,{\rm d}x + \Delta_\rho r_\ell^{(j,k)}\right).
\end{gather}
The constant term of the right-hand side of~\eqref{427} at $x=0$
is equal to
\begin{gather}
 \eta \left(\eta^{-1} \partial_{\rho} +\frac{1}{2} \eta^{-2}\partial_{\rho}^2 +\cdots \right) r_{-1}^{(j,k)}+\left(\eta^{-1}\partial_{\rho} +\frac{1}{2} \eta^{-2}\partial_{\rho}^2+\cdots \right) r_0^{(j,k)}+\cdots\nonumber\\
\qquad{}=\partial_{\rho}r_{-1}^{(j,k)}+ \left( \frac{1}{2} \partial_{\rho}^2r_{-1}^{(j,k)} +\partial_{\rho} r_0^{(j,k)} \right)\eta^{-1}+\cdots.\label{22}
\end{gather}
Comparing the coefficients of $\eta^{-m}$, $m=0,1$, of \eqref{ddef} and the right-hand side of \eqref{22}, we have
\begin{gather}
d_0^{(j,k)} = \partial _{\rho}r^{(j,k)}_{-1},\nonumber\\
d_1^{(j,k)} = \frac{1}{2} \partial_{\rho}^2r_{-1}^{(j,k)} +\partial_{\rho} r_0^{(j,k)}.\label{djk}
\end{gather}
By the definition of $V^{(j,k)}_0$, we have
\begin{gather}\label{tempv0jk}
\Delta_\rho V^{(j,k)}_0=\frac{1}{2} \lim_{x \to 0} \left(\int_{\tilde\gamma_{j,k}} \Delta_\rho S\,{\rm d}x-\eta \int_{\tilde\gamma_{j,k}} \Delta_\rho S_{-1} \,{\rm d}x-\int_{\tilde\gamma_{j,k}} \Delta_\rho S_0 \,{\rm d}x\right).
\end{gather}
By the definition of $c_0^{(j,k)}$ and $r^{(j,k)}_{\ell}$, $\ell=-1,0$, the right-hand side of \eqref{tempv0jk} is written in the form:
\begin{gather}\label{temp1v0jk}
\Delta_\rho V^{(j,k)}_0=\frac{1}{2} \left( c_0^{(j,k)} -\eta \Delta_\rho r^{(j,k)}_{-1} -\Delta_\rho r^{(j,k)}_{0} \right).
\end{gather}
Differentiating \eqref{temp1v0jk} with respect to $\rho$ and using \eqref{djk}, we have
\begin{gather*}
\partial_\rho \Delta_\rho V^{(j,k)}_0=\frac{1}{2}\left( \partial_\rho c_0^{(j,k)} - \Delta_\rho \left( \eta d_0^{(j,k)}+d_1^{(j,k)}-\frac{1}{2}\partial_\rho d_0^{(j,k)} \right) \right).
\end{gather*}
As we remarked above, $c_0^{(j,k)}$ is written as the logarithm of a product of several linear functions in $\rho$.
Setting
\[
g_\rho^{(j,k)}=-\frac{1}{2}\left( \eta d_0^{(j,k)}+d_1^{(j,k)}-\frac{1}{2}\partial_\rho d_0^{(j,k)} \right),
\] we have the lemma for $\varrho=0$. Similarly, we get the differential-difference equations for Voros coefficients at the infinity.
\end{proof}

\begin{Remark}
If $N=3$, we have
\begin{gather*}
2\partial_{a_{i,1}}\Delta_{a_{i,1}}V_0^{(1, 2)}
 =-\frac{\eta }{a_i-b_1+1}+\frac{\eta }{a_i-b_2+1}\\
\qquad{}
+\Delta_{a_{i,1}}\left(-\eta \log \left(\frac{a_{i,1}-b_{2,1}}{a_{i,1}-b_{1,1}}\right)+\frac{2 a_{i,0}-2 b_{1,0}+1}{2 a_{i,1}-2 b_{1,1}}-\frac{2 a_{i,0}-2b_{2,0}+1}{2 a_{i,1}-2 b_{2,1}}\right),
\\
2\partial_{a_{i,1}}\Delta_{a_{i,1}}V_0^{(j, 3)} =\frac{\eta }{a_i-b_j+1}+\frac{\eta }{a_i}\\
\qquad{}
+\Delta_{a_{i,1}}\left(-\eta\log \left(\frac{a_{i,1}}{a_{i,1}-b_{j,1}}\right)-\frac{b_{j,0}-1}{ a_{i,1}-b_{j,1}}+\frac{\left(1-2 a_{i,0}\right) b_{j,1}}{2a_{i,1}(a_{i,1}-b_{j,1})} \right),
\\
2\partial_{b_{m,1}}\Delta_{b_{m,1}}V_0^{(1, 2)}\\
\qquad{}
 =(-1)^{m-1}\left( \sum_{i=1}^3 \frac{\eta }{b_m-a_i}-\sum_{\ell \in \{1,2\} \setminus \{m\}}\left( \frac{\eta }{b_m-1}-\frac{\eta }{b_m-b_\ell}-\frac{\eta }{b_m-b_\ell+1}\right)\right)\\
\qquad\quad{}
+\Delta_{b_{m,1}}(-1)^{m-1}\Bigg({-}\eta \log \left(\frac{\left(a_{1, 1}-b_{m, 1}\right) \left(a_{2, 1}-b_{m, 1}\right)\left(a_{3, 1}-b_{m, 1}\right)}{b_{m, 1}
 \left(b_{1, 1}-b_{2, 1}\right){}^2}\right)\\
\qquad\quad{}
-\frac{1}{2} \left(\sum_{i=1}^3\frac{2( a_{i, 0}- b_{m, 0})+1}{a_{i, 1}-b_{m, 1}}+\frac{3-2 b_{m, 0}}{b_{m, 1}}-\frac{4\left(b_{1, 0}-b_{2, 0}\right)}{b_{1, 1}-b_{2, 1}} \right)\Bigg),
\\
2\partial_{b_{m,1}}\Delta_{b_{m,1}}V_0^{(j, 3)} =\frac{\eta }{b_m-b_j+1}-\frac{\eta }{b_m}\\
\qquad{} +\Delta_{b_{m,1}}(-1)^{m-1}\left(-\eta\log\left(1-\frac{b_{j,1}}{b_{m,1}}\right)-\frac{ b_{j,0}-1}{b_{j, 1}-b_{m, 1}}
+\frac{ (2b_{m,0}-1)b_{j,1}}{2 b_{m, 1} \left(b_{j, 1}-b_{m, 1}\right)}\right),
\\
2\partial_{b_{m,1}}\Delta_{b_{m,1}}V_0^{(m, 3)} =\sum_{i=1}^3\frac{\eta }{b_m-a_i}
-\frac{\eta }{b_m-1}-\frac{\eta }{b_m}-\frac{\eta }{b_m-b_j}\\
\qquad{} +\Delta_{b_{m,1}}\Bigg({-}\eta \log \left(\frac{\left(a_{1, 1}-b_{m, 1}\right) \left(a_{2, 1}-b_{m, 1}\right)\left(a_{3, 1}-b_{m, 1}\right)}{b_{m, 1}^2
 \left(b_{m, 1}-b_{j, 1}\right)}\right)\\
\qquad{} -\frac{1}{2} \left(\sum_{i=1}^3\frac{2( a_{i, 0}- b_{m, 0})+1}{a_{i, 1}-b_{m, 1}}+\frac{4(1- b_{m, 0})}{b_{m, 1}}+\frac{2( b_{j, 0}-
 b_{m, 0})+1}{b_{m, 1}-b_{j, 1}}\right)\Bigg)
\end{gather*}
for $i=1,2,3$, $m=1,2$ and $j \in \{1,2\} \setminus \{m\}$.
\end{Remark}

The system of differential-difference equations given in Lemma~\ref{ddeqlem} characterizes $V_\varrho^{(j,k)}$, for we have the following lemma:

\begin{Lemma}\label{homovoros}
Let ${W}=\sum_{n=-1}^\infty w_n \eta^{-n}$ be a formal solution of the differential-difference equation
\begin{gather}\label{seizi}
\partial_{\rho} \Delta_{\rho}W=0
\end{gather}
for $\rho= a_{1,1},\dots, a_{N,1}, b_{1,1}, \dots,b_{N-1,1}$ and
suppose that $w_n$ is a homogeneous function of deg\-ree~$(-n)$ with respect to $a_{1,1},\dots, a_{N,1}, b_{1,1}, \dots,b_{N-1,1}$ and $w_{-1}=w_{0}=0$.
Then, we have
\begin{gather*}
W=0.
\end{gather*}
\end{Lemma}
\begin{proof}
By the definition of $\Delta_{\rho}$, we have
\begin{gather*}
\partial_{\rho} \Delta_{\rho} W
=\sum_{\ell=0}^\infty\left(\sum_{n=-1}^{\ell-1}\frac{1}{(\ell-n)!}\partial_{\rho}^{\ell-n+1} w_n \right)\eta^{-\ell}
\end{gather*}
for all $\rho \in \{ a_{1,1},\dots, a_{N,1}, b_{1,1}, \dots,b_{N-1,1} \}$. Comparing the coefficients for both sides of~\eqref{seizi} with respect to the powers of $\eta$, we have
\begin{gather*}
\sum_{n=-1}^{\ell-1}\frac{1}{(\ell-n)!}\partial_{\rho}^{\ell-n+1} w_n=0,\qquad n=0,1,2,\dots.
\end{gather*}
Since $w_{-1}=w_{0}=0$, we have $\partial_\rho^2 w_n=0$ for $n=1,2,\dots$ recursively.
By the assumption, $w_n$~is a~homogeneous function of degree $(-n)$ with respect to $a_{1,1},\dots, a_{N,1}, b_{1,1}, \dots,b_{N-1,1}$. Therefore we obtain
\begin{gather*}
w_{n} =0, \qquad n=1, 2,\dots.\tag*{\qed}
\end{gather*}\renewcommand{\qed}{}
\end{proof}

\subsection{Voros coefficients}
We have the explicit forms of the Voros coefficients $V_{\varrho}^{(j,k)}$ at $\varrho$, $\varrho=0,\infty$.
\begin{Theorem}\label{thm1}
Let $\varrho$ be $0$ or $\infty$. If there exists a simple turning point of type $(j,k)$, $j<k$, of~\eqref{nFn-1eq}, then the Voros coefficients $V_{\varrho}^{(j,k)}$ are written in the form:
\begin{gather*}
\frac{1}{2} \sum_{\ell=2}^{\infty} \frac{(-1)^{\ell+1} \eta^{1-\ell}}{\ell(\ell-1)} V^{(j,k)}_{\varrho,\ell},
\end{gather*}
where
\begin{gather*}
V^{(j,k)}_{0,\ell}=
\frac{B_\ell(b_{j,0}-1)}{b_{j,1}^{\ell-1}}
-\frac{B_\ell(b_{k,0}-1)}{b_{k,1}^{\ell-1}}+\sum_{i=1}^N \frac{B_\ell(b_{k,0}-a_{i,0})}{(b_{k,1}-a_{i,1})^{\ell-1}}\\
\hphantom{V^{(j,k)}_{0,\ell}=}{}
-\sum_{i=1}^N \frac{B_\ell(b_{j,0}-a_{i,0})}{(b_{j,1}-a_{i,1})^{\ell-1}}+\sum_{\substack{m=1 \\ m\not=j}}^{N-1} \frac{B_\ell(b_{j,0}-b_{m,0})}{(b_{j,1}-b_{m,1})^{\ell-1}}-\sum_{\substack{m=1 \\ m\not=k}}^{N-1} \frac{B_\ell(b_{k,0}-b_{m,0})}{(b_{k,1}-b_{m,1})^{\ell-1}},
\qquad k\not= N,\\
V^{(j,N)}_{0,\ell}=\frac{B_\ell(b_{j,0}-1)}{b_{j,1}^{\ell-1}}-\sum_{i=1}^N \frac{B_\ell(a_{i,0})}{a_{i,1}^{\ell-1}}
-\sum_{i=1}^N \frac{B_\ell(b_{j,0}-a_{i,0})}{(b_{j,1}-a_{i,1})^{\ell-1}}\\
\hphantom{V^{(j,N)}_{0,\ell}=}{}
+\sum_{\substack{m=1 \\ m\not=j}}^{N-1}\frac{B_\ell(b_{j,0}-b_{m,0})}{(b_{j,1}-b_{m,1})^{\ell-1}}+\sum_{m=1}^{N-1} \frac{B_\ell(b_{m,0})}{b_{m,1}^{\ell-1}},
\\
V^{(j,k)}_{\infty,\ell}=\sum_{m=1}^{N-1} \left( \frac{B_\ell(b_{m,0}-a_{k,0})}{(b_{m,1}-a_{k,1})^{\ell-1}}-\frac{B_\ell(b_{m,0}-a_{j,0})}{(b_{m,1}-a_{j,1})^{\ell-1}} \right)\\
\hphantom{V^{(j,k)}_{\infty,\ell}=}{}
+\left(\frac{B_\ell(a_{j,0})}{a_{j,1}^{\ell-1}}+\sum_{\substack{m=1 \\ m\not=j}}^{N-1}\frac{B_\ell(a_{i,0}-a_{j,0})}{(a_{i,1}-a_{j,1})^{\ell-1}}\right)-\left(\frac{B_\ell(a_{k,0})}{a_{k,1}^{\ell-1}}+\sum_{\substack{m=1 \\ m\not=k}}^{N-1}\frac{B_\ell(a_{i,0}-a_{k,0})}{(a_{i,1}-a_{k,1})^{\ell-1}} \right).
\end{gather*}
Here $B_\ell(t)$ denotes the $\ell$-th Bernoulli polynomial defined by
\begin{gather*}
\frac{x {\rm e}^{xt}}{{\rm e}^x-1}=\sum_{\ell=0}^\infty \frac{B_\ell(t)}{\ell!}x^\ell.
\end{gather*}
\end{Theorem}

\begin{proof}
To obtain the explicit forms of $V_\varrho^{(j,k)}$, we will solve~\eqref{ddeq}.
We may write $f_\rho^{(j,k)}$ in the form
\begin{gather}\label{fcondi}
f_\rho^{(j,k)} = \frac{1}{2}\sum_m \frac{1}{ u^{(j,k)}_{\rho, m}+ v^{(j,k)}_{\rho, m} \eta^{-1} },
\end{gather}
where $u^{(j,k)}_{\rho, m}$ and $v^{(j,k)}_{\rho, m} $ being linear functions of~$\rho$ and~$\rho_0$, respectively, which are independent of~$\eta$.
The explicit forms of $u^{(j,k)}_{\rho, m}$ and $v^{(j,k)}_{\rho, m} $ are given by Lemma~\ref{ddeqlem}.
The coefficient of~$\rho$ of~$u^{(j,k)}_{\rho, m}$ is equal to~$1$ or~$-1$, which denotes $ \epsilon^{(j,k)}_{\rho, m}$ $(=\pm1)$.
The right-hand side of~\eqref{fcondi} is obtained from~$u^{(j,k)}_{\rho, m}$ by the shift:
\begin{gather*}
\frac{1}{2}\sum_m \frac{1}{ u^{(j,k)}_{\rho, m}+ v^{(j,k)}_{\rho, m} \eta^{-1} }=\frac{1}{2}\sum_m \frac{1}{u^{(j,k)}_{\rho, m}}\bigg|_{\rho \mapsto \rho+\epsilon^{(j,k)}_{\rho, m} v^{(j,k)}_{\rho, m}\eta^{-1} }.
\end{gather*}
Thus \eqref{ddeq} is written in the following form:
\begin{gather*}
\partial_\rho \Delta_\rho V_\varrho^{(j,k)}=\frac{1}{2}\sum_m {\rm e}^{ \epsilon^{(j,k)}_{\rho, m} v_{\rho, m}^{(j,k)} \eta^{-1}\partial_{\rho}}\frac{1}{u_{\rho, m}^{(j,k)}}+\Delta_{\rho}g^{(j,k)}_\rho.
\end{gather*}
Hence \eqref{ddeq} can be solved by using a formal differential operator of infinite order:
\begin{align*}
V_\varrho^{(j,k)}&=\frac{1}{2}\sum_m \frac{{\rm e}^{\epsilon^{(j,k)}_{\rho, m} v_{\rho, m}^{(j,k)} \eta^{-1}\partial_{\rho}}}{ \big({\rm e}^{ \eta^{-1} \partial_\rho} -1\big) } \partial_{\rho}^{-1}\frac{1}{u_{\rho, m}^{(j,k)}}+\partial_\rho^{-1} g^{(j,k)}_\rho\\
&=\frac{1}{2}\sum_{\ell=0}^\infty \sum_m \frac{B_\ell\big(\epsilon^{(j,k)}_{\rho, m} v_{\rho, m}^{(j,k)}\big)}{\ell!}\eta^{1-\ell} \partial_{\rho}^{\ell-2} \frac{1}{u_{\rho, m}^{(j,k)}}+\partial_\rho^{-1} g^{(j,k)}_\rho.
\end{align*}
Using Lemma \ref{homovoros}, we have
\begin{gather*}
V_\varrho^{(j,k)}=\frac{1}{2}\sum_{\ell=2}^\infty \sum_m \frac{B_\ell\big(\epsilon^{(j,k)}_{\rho, m} v_{\rho, m}^{(j,k)}\big)}{\ell!}\eta^{1-\ell} \partial_{\rho}^{\ell-2} \frac{1}{u_{\rho, m}^{(j,k)}}+C_\rho(\check\rho),
\end{gather*}
where $C_\rho(\check\rho)$ is an arbitrary function of $\{a_{1,1}, \dots, a_{N,1}, b_{1,1}, \dots ,b_{N-1,1} \} \setminus \{\rho\}$.
We repeat the above discussion for every $\rho \in \{a_{1,1}, \dots, a_{N,1}, b_{1,1}, \dots ,b_{N-1,1} \}$. Then we can determine $C_\rho(\check\rho)$ by adjustment. Thus we have the theorem.
\end{proof}

\begin{Remark}
If $N=3$, we have
\begin{gather*}
V^{(1,2)}_{0,\ell}=\sum_{\substack{i=1,2,3\\j=1,2}} (-1)^j \frac{B_\ell(b_{j,0}-a_{i,0})}{(b_{j,1}-a_{i,1})^{\ell-1}}\\
\hphantom{V^{(1,2)}_{0,\ell}=}{}
+\frac{B_\ell(b_{1,0}-1)}{b_{1,1}^{\ell-1}}-\frac{B_\ell(b_{2,0}-1)}{b_{2,1}^{\ell-1}}+\frac{ B_\ell(b_{1,0}-b_{2,0}+1)+B_\ell(b_{1,0}-b_{2,0})}{(b_{1,1}-b_{2,1})^{\ell-1}},
\\
V^{(1,3)}_{0,\ell} =-\sum_{i=1,2,3} \frac{B_\ell(b_{1,0}-a_{i,0})}{(b_{1,1}-a_{i,1})^{\ell-1}}-\sum_{i=1,2,3} \frac{B_\ell(a_{i,0})}{a_{i,1}^{\ell-1}}\\
\hphantom{V^{(1,3)}_{0,\ell} =}{}
+\frac{B_\ell(b_{1,0})+B_\ell(b_{1,0}-1)}{b_{1,1}^{\ell-1}}+\frac{B_\ell(b_{2,0})}{b_{2,1}^{\ell-1}}+\frac{B_\ell(b_{1,0}-b_{2,0})}{(b_{1,1}-b_{2,1})^{\ell-1}},
\\
V^{(2,3)}_{0,\ell}=-\sum_{i=1,2,3} \frac{B_\ell(b_{2,0}-a_{i,0})}{(b_{2,1}-a_{i,1})^{\ell-1}}-\sum_{i=1,2,3} \frac{B_\ell(a_{i,0})}{a_{i,1}^{\ell-1}}\\
\hphantom{V^{(2,3)}_{0,\ell}=}{}
+\frac{B_\ell(b_{2,0})+B_\ell(b_{2,0}-1)}{b_{2,1}^{\ell-1}}+\frac{B_\ell(b_{1,0})}{b_{1,1}^{\ell-1}}+\frac{B_\ell(b_{2,0}-b_{1,0})}{(b_{2,1}-b_{1,1})^{\ell-1}},
\\
V^{(j,k)}_{\infty,\ell}=\sum_{m=1,2} \left( \frac{B_\ell(b_{m,0}-a_{k,0})}{(b_{m,1}-a_{k,1})^{\ell-1}}-\frac{B_\ell(b_{m,0}-a_{j,0})}{(b_{m,1}-a_{j,1})^{\ell-1}} \right)\\
\hphantom{V^{(j,k)}_{\infty,\ell}=}{}
+\left(\frac{B_\ell(a_{j,0})}{a_{j,1}^{\ell-1}}+\sum_{\substack{i=1 \\ i\not=j}}^{3}\frac{B_\ell(a_{i,0}-a_{j,0})}{(a_{i,1}-a_{j,1})^{\ell-1}}\right)-\left(\frac{B_\ell(a_{k,0})}{a_{k,1}^{\ell-1}}+\sum_{\substack{i=1 \\ i\not=k}}^{3}\frac{B_\ell(a_{i,0}-a_{k,0})}{(a_{i,1}-a_{k,1})^{\ell-1}} \right).
\end{gather*}
\end{Remark}

We consider the Borel summability of $V_\varrho^{(j,k)}$.
We set
\begin{gather*}
\tilde{V} (\kappa,\kappa_0)= \frac{1}{2}\sum_{\ell=2}^\infty \frac{(-1)^{\ell+1} \eta^{1-\ell}}{\ell(\ell-1)} \frac{B_\ell(\kappa_0)}{\kappa^{\ell-1}}.
\end{gather*}
As we saw in the preceding theorem, $V_\varrho^{(j,k)}$ are expressed as a finite sum of the formal series of the form $\tilde{V} (\kappa,\kappa_0)$ for suitable choice of~$\kappa$,~$\kappa_0$.
Hence it is sufficient to consider the Borel summability of $\tilde{V} (\kappa,\kappa_0)$. Its summability is well known: the formal series $\tilde{V} (\kappa,\kappa_0)$ of $\eta^{-1}$ is Borel summable if $\operatorname{Re}(\kappa)\neq 0$ and the Borel sum of $\tilde{V} (\kappa,\kappa_0)$ equals
\begin{gather*}
 \frac{1}{2}\log\left( \frac{\sqrt{2 \pi} (\kappa \eta)^{\kappa_0+\kappa\eta-\frac{1}{2}}}{{\rm e}^{\kappa \eta} \Gamma(\kappa_0+\kappa\eta)}\right), \qquad \operatorname{Re}(\kappa)>0,\\
\frac{1}{2} \log \left(\frac{(-\kappa \eta)^{\kappa_0+\kappa\eta-\frac{1}{2}}\Gamma\left(1-(\kappa_0+\kappa\eta) \right)}{ \sqrt{2 \pi} {\rm e}^{\kappa \eta}}\right), \qquad \operatorname{Re}(\kappa)<0,
\end{gather*}
(see \cite[Lemma 4.6]{ATT4}).
This expression is obtained by the Binet formula~\cite{EMOT}.
Thus, we have
\begin{Corollary}\label{vsd}
Let $\varrho$ be $0$ or $\infty$. Also let
\begin{gather*}
D(\kappa)=\{(a_{1,1},\dots,a_{N,1},b_{1,1},\dots,b_{N-1,1}) \,|\, \operatorname{Re}(\kappa)\not =0 \},
\end{gather*}
where $\kappa$ is a linear homogeneous function of $a_{1,1},\dots,a_{N,1}, b_{1,1},\dots,b_{N-1,1}$.
The Voros coefficient~$V_\varrho^{(j,k)}$ of type $(j,k)$, $j<k$, is Borel summable if
$(a_{1,1},\dots,a_{N,1},b_{1,1},\dots,b_{N-1,1})$ belongs to~$D^{(j,k)}_\varrho$. Here we set
\begin{gather*}
D^{(j,k)}_0= D(b_{j,1})\cap D(b_{k,1}) \cap \bigcap_{i=1}^N{D(b_{j,1}-a_{i,1})} \cap \bigcap_{\substack{m=1 \\ m\not=j}}^{N-1}{D(b_{j,1}-b_{m,1})}\\
\hphantom{D^{(j,k)}_0=}{}
\cap \bigcap_{i=1}^N{D(b_{k,1}-a_{i,1})}\cap \bigcap_{\substack{m=1 \\ m\not=k}}^{N-1}{D(b_{k,1}-b_{m,1})}, \qquad k\not= N,
\\
D^{(j,N)}_0= \bigcap_{i=1}^N{D(a_{i,1})} \cap \bigcap_{m=1}^{N-1} {D(b_{m,1})}\cap \bigcap_{i=1}^N{D(b_{j,1}-a_{i,1})} \cap \bigcap_{\substack{m=1 \\ m\not=j}}^{N-1}{D(b_{j,1}-b_{m,1})},
\\
D^{(j,k)}_\infty= D(a_{j,1}) \cap D(a_{k,1})\cap \bigcap_{m=1}^{N-1} {D(b_{m,1}-a_{k,1})}\cap \bigcap_{m=1}^{N-1} {D(b_{m,1}-a_{j,1})}\\
\hphantom{D^{(j,k)}_\infty=}{}
\cap \bigcap_{\substack{i=1 \\ i\not=j}}^{N} {D(a_{i,1}-a_{j,1})}\cap \bigcap_{\substack{i=1 \\ i\not=k}}^{N} {D(a_{i,1}-a_{k,1})} .
\end{gather*}
\end{Corollary}

\section{Concluding remarks and future problems}\label{section4}
We have obtained explicit forms of the Voros coefficients at the origin and at the infinity (Theo\-rem~\ref{thm1}). To give an explicit form of the Voros coefficient at $x = 1$ is our future problem.
The difficulty of the problem comes from the multiplicity of the characteristic exponents at $x=1$.

As in the case of the Gauss hypergeometric differential equation~\cite{AT} and of its confluent families~\cite{KoT,T}, we may get the formulas describing the parametric Stokes phenomena for WKB solutions of~\eqref{nFn-1eq} by using Corollary~\ref{vsd} under the assumption that their formal solutions are Borel summable in some region. The Borel summability of WKB solutions of our equation is also a future problem.

\subsection*{Acknowledgements}
The first author is supported by JSPS KAKENHI Grant No.~18K03385. The authors would like to thank the referees for their careful readings and detailed comments on this article.

\pdfbookmark[1]{References}{ref}
\LastPageEnding

\end{document}